\documentclass[12pt]{amsart}
\usepackage{amsmath,amssymb,amsfonts,enumerate,amsthm, amscd,}
\usepackage{graphicx}
\usepackage{color}
\usepackage[colorlinks=true,linkcolor=blue,urlcolor=blue,citecolor   = red]{hyperref}

\newtheorem{theorem}{Theorem}[section]
\newtheorem{cor}[theorem]{Corollary}
\newtheorem{lem}[theorem]{Lemma}
\newtheorem{proposition}[theorem]{Proposition}

\newtheorem{rem}{Remark}

\begin{document}
\title[Some properties on the uaw-convergence in Banach lattices]{Some
properties on the unbounded absolute weak convergence in Banach lattices}
\author[A. Elbour]{Aziz Elbour}
\address[A. Elbour]{Department of Mathematics, Faculty of Science and
Technology, Moulay Ismail University of Meknes, P.O. Box 509, Boutalamine
52000, Errachidia, Morocco}
\email{azizelbour@hotmail.com}

\begin{abstract}
In this paper, we investigate more about relationship between $uaw$%
-convergence (resp. $un$-convergence) and the weak convergence. More
precisely, we characterize Banach lattices on which every weak null sequence
is $uaw$-null. Also, we characterize order continuous Banach lattices under
which every norm bounded $un$-null net (resp. sequence) is weakly null. As a
consequence, we study relationship between sequentially $uaw$-compact
operators and weakly compact operators. Also, it is proved that every
continuous operator, from a Banach lattice $E$ into a non-zero Banach space $%
X$, is unbounded continuous if and only if $E^{\prime }$ is order
continuous. Finally, we give a new characterization of $b$-weakly compact
operators using the $uaw$-convergence sequences.
\end{abstract}

\maketitle

\section{Introduction}

The unbounded convergences received much attention. Several papers about
\textquotedblleft \emph{unbounded order convergence, unbounded norm
convergence}\textquotedblright\ and \emph{\textquotedblleft unbounded
absolute weak convergence}\textquotedblright\ have been published, see \cite%
{Deng, Gao1, Gao2, Kan17,Zabeti18}. Throughout this paper, $E$ will stand
for a Banach lattice. A net $(x_{\alpha })$ in $E$ is said to be:

\begin{itemize}
\item \emph{unbounded norm convergent} ($un$-convergent, for short) if for
each $u\in E^{+}$, $\Vert |x_{\alpha }-x|\wedge u\Vert \rightarrow 0$, in
notation $x_{\alpha }\overset{un}{\longrightarrow }x$.

\item \emph{unbounded absolute weak convergent} ($uaw$-convergent, in brief)
to $x\in E$ if for each $u\in E^{+}$, $|x_{\alpha }-x|\wedge u\rightarrow 0$
weakly, in notation $x_{\alpha }\overset{uaw}{\longrightarrow }x$.
\end{itemize}

Note that both convergences are topological \cite{Deng, Zabeti18}. Clearly, $%
un$-convergence implies the $uaw$-convergence, but the converse is false in
general. However, $un$-convergence agrees with the $uaw$-convergence iff $E$
is order continuous \cite[Theorem 4]{Zabeti18}.

Moreover, several investigation on the relationship between these
convergences and the weak convergence are given in \cite{Deng, Kan17,
Zabeti18}. It is proved in \cite[Theorem 7]{Zabeti18} that $E^{\prime }$ is
order continuous iff every norm bounded $uaw$-null net (resp. sequence) is
weak null. In particular, if $E^{\prime }$ is order continuous then every
norm bounded $un$-null net is weak null \cite[Theorem 2.4]{Deng}. Also, it
is proved that $E$ is discrete and order continuous iff every weak null net
(resp. sequence) is $un$-null \cite[Proposition 4.16]{Kan17} (see also \cite[%
Proposition 6.2]{Deng}).

In this paper, we investigate more about relationship between $uaw$-
convergence (resp. $un$-convergence) and the weak convergence. More
precisely, we characterize Banach lattices on which every weak null sequence
is $uaw$-null (Theorem \ref{sequence}). Also, we characterize order
continuous Banach lattices under which every norm bounded $un$-null net
(resp. sequence) is weakly null (Proposition \ref{un-w}). Furthermore, we
show that for norm bounded sequence, the $uaw$-convergence, $un$-convergence
and the weak convergence coincide on a Banach lattice $E$ iff $E$ is
discrete and both $E$ and $E^{\prime }$ are order continuous (corollaries %
\ref{cor2}, and \ref{cor3}). Next, we discuss when weakly compact operators
are sequentially $uaw$-compact (Proposition \ref{weaklycompact} and Theorem %
\ref{sequentially}). Also, it is proved that every continuous operator, from
a Banach lattice $E$ into a non-zero Banach space $X$, is (sequentially)
unbounded continuous iff $E^{\prime }$ is order continuous (Theorem \ref%
{unbounded-cont}). Finally, we give a new characterization of $b$-weakly
compact operators using the $uaw$-convergence sequences (Theorem \ref%
{b-weaklycompact}).

We refer to \cite{Alip, Alip2, MN} for all unexplained terminology and
standard facts on vector and Banach lattices.

\section{Main Results}

We start this section by the following result which tells us when weak
convergent sequences are $uaw$-convergent.

\begin{theorem}
\label{sequence}Let $E$ be a Banach lattice. The following are equivalent:

\begin{enumerate}
\item The lattice operations of $E$ are sequentially weakly continuous;

\item For any sequence $(x_{n})\subset E$, $x_{n}\overset{w}{\longrightarrow
}0$ implies $x_{n}\overset{uaw}{\longrightarrow }0$.
\end{enumerate}
\end{theorem}

\begin{proof}
$\left( 1\right) \Rightarrow \left( 2\right) $ Obvious because $\left\vert
x_{n}\right\vert \overset{w}{\longrightarrow }0$ implies $x_{n}\overset{uaw}{%
\longrightarrow }0$.

$\left( 2\right) \Rightarrow \left( 1\right) $ Assumes $(2)$ holds. Let $%
(x_{n})\subset E$ be a weakly null sequence. So by hypothesis, $x_{n}\overset%
{uaw}{\longrightarrow }0$. We need to show that $\left\vert x_{n}\right\vert
\overset{w}{\longrightarrow }0$. To this end, let $x^{\prime }\in \left(
E^{\prime }\right) ^{+}$ and $\varepsilon >0$. By Theorem 4.37 of \cite{Alip}%
, there exists $u\in E^{+}$ such that, for all $n$,%
\begin{equation*}
x^{\prime }\left( \left\vert x_{n}\right\vert -\left\vert x_{n}\right\vert
\wedge u\right) =x^{\prime }\left( \left\vert x_{n}\right\vert -u\right)
^{+}<\varepsilon .
\end{equation*}

Since $x^{\prime }\left( \left\vert x_{n}\right\vert \wedge u\right)
\longrightarrow 0$, we conclude that $x^{\prime }\left( \left\vert
x_{n}\right\vert \right) \longrightarrow 0$.

Therefore, $\left\vert x_{n}\right\vert \overset{w}{\longrightarrow }0$, as
desired.
\end{proof}

We know that the lattice operations in every AM-space are sequentially
weakly continuous. Also, if $E$ is a Banach lattice with an order continuous
norm then the lattice operations of $E$ are sequentially weakly continuous
if and only if $E$ is discrete (see Proposition 2.5.23 of \cite{MN} and
Corollary 2.3 of \cite{Chen}).

\begin{cor}
\label{cor1}For a Banach lattice $E$, the following are equivalent:

\begin{enumerate}
\item The lattice operations of $E$ are sequentially weakly continuous and $%
E^{\prime }$ is order continuous;

\item $x_{n}\overset{w}{\longrightarrow }0$ $\Leftrightarrow $ $x_{n}\overset%
{uaw}{\longrightarrow }0$ for every norm bounded sequence $(x_{n})\subset E$.
\end{enumerate}
\end{cor}

\begin{proof}
Follows immediately from \cite[Theorem 7]{Zabeti18} and Theorem \ref%
{sequence}.
\end{proof}

\begin{cor}
For a\ norm bounded sequence $(x_{n})$ in every AM-space, we have $x_{n}%
\overset{w}{\longrightarrow }0\Leftrightarrow x_{n}\overset{uaw}{%
\longrightarrow }0$.
\end{cor}

It is easy to see from \cite[Theorem 4]{Zabeti18} and its proof that a
Banach lattice $E$ is order continuous iff $x_{n}\overset{uaw}{%
\longrightarrow }0$ $\Leftrightarrow $ $x_{n}\overset{un}{\longrightarrow }0$
for every norm bounded (resp. order bounded) sequence $(x_{n})\subset E$.

\begin{cor}
\label{cor2}For a Banach lattice $E$, the following are equivalent:

\begin{enumerate}
\item $E$ is discrete and both $E$ and $E^{\prime }$ are order continuous;

\item $x_{n}\overset{w}{\longrightarrow }0$ $\Leftrightarrow $ $x_{n}\overset%
{uaw}{\longrightarrow }0$ $\Leftrightarrow $ $x_{n}\overset{un}{%
\longrightarrow }0$ for every norm bounded sequence $(x_{n})\subset E$.
\end{enumerate}
\end{cor}

\begin{proof}
Follows immediately from \cite[Theorem 4]{Zabeti18}, \cite[Corollary 2.3]%
{Chen}, and Corollary \ref{cor1}.
\end{proof}

\begin{proposition}
\label{un-w}For an order continuous Banach lattice $E$, the following are
equivalent:

\begin{enumerate}
\item $E^{\prime }$ is order continuous;

\item $x_{\alpha }\overset{un}{\longrightarrow }0\Rightarrow x_{\alpha }%
\overset{w}{\longrightarrow }0$ for every norm bounded net $(x_{\alpha
})\subset E$.

\item $x_{n}\overset{un}{\longrightarrow }0\Rightarrow x_{n}\overset{w}{%
\longrightarrow }0$ for every norm bounded sequence $(x_{n})\subset E$.
\end{enumerate}
\end{proposition}

\begin{proof}
$\left( 1\right) \Rightarrow \left( 2\right) $ Follows from \cite[Theorem 6.4%
]{Deng}.

$\left( 2\right) \Rightarrow \left( 3\right) $ Obvious.

$\left( 3\right) \Rightarrow \left( 1\right) $ Let $(x_{n})\subset E^{+}$ be
any norm bounded disjoint sequence and let $u\in E^{+}$. Since $E$ is order
continuous, it follows from \cite[Proposition 3.5]{Kan17} that $x_{n}\overset%
{un}{\longrightarrow }0$. So, by hypothesis $x_{n}\overset{w}{%
\longrightarrow }0$. Now Theorem 116.1 of \cite{Z2} (see also \cite[Theorem
2.4.14]{MN}) finish the proof.
\end{proof}

\begin{cor}
\label{cor3}For a Banach lattice $E$, the following assertions are
equivalent:

\begin{enumerate}
\item $E$ is discrete and both $E$ and $E^{\prime }$ are order continuous;

\item $x_{\alpha }\overset{un}{\longrightarrow }0\Leftrightarrow x_{\alpha }%
\overset{uaw}{\longrightarrow }0\Leftrightarrow x_{\alpha }\overset{w}{%
\longrightarrow }0$ for every norm bounded net $(x_{\alpha })\subset E$.

\item $x_{n}\overset{w}{\longrightarrow }0$ $\Leftrightarrow $ $x_{n}\overset%
{uaw}{\longrightarrow }0$ $\Leftrightarrow $ $x_{n}\overset{un}{%
\longrightarrow }0$ for every norm bounded sequence $(x_{n})\subset E$.

\item $x_{\alpha }\overset{un}{\longrightarrow }0\Leftrightarrow x_{\alpha }%
\overset{w}{\longrightarrow }0$ for every norm bounded net $(x_{\alpha
})\subset E$.

\item $x_{n}\overset{un}{\longrightarrow }0\Leftrightarrow x_{n}\overset{w}{%
\longrightarrow }0$ for every norm bounded sequence $(x_{n})\subset E$.
\end{enumerate}
\end{cor}

\begin{proof}
$\left( 1\right) \Rightarrow \left( 2\right) $ Follows from Theorems 4 and 7
of \cite{Zabeti18} and \cite[Proposition 4.16]{Kan17}.

$\left( 2\right) \Rightarrow \left( 4\right) \Rightarrow \left( 5\right) $
and $\left( 2\right) \Rightarrow \left( 3\right) \Rightarrow \left( 5\right)
$ are obvious.

$\left( 5\right) \Rightarrow \left( 1\right) $ Follows from \cite[%
Proposition 4.16]{Kan17}\ and Proposition \ref{un-w} by noting that every
weak null sequence is norm bounded.
\end{proof}

\begin{rem}
It follows from \cite[Theorem 4]{Zabeti18} (resp. Corollary \ref{cor1}) that
the condition (3) of Corollary \ref{cor3} cannot be replaced by the
following condition%
\begin{equation*}
x_{n}\overset{uaw}{\longrightarrow }0\Leftrightarrow x_{n}\overset{un}{%
\longrightarrow }0,
\end{equation*}

(resp.
\begin{equation*}
x_{n}\overset{w}{\longrightarrow }0\Leftrightarrow x_{n}\overset{uaw}{%
\longrightarrow }0\text{)}.
\end{equation*}
\end{rem}

\begin{proposition}
\label{almostorder}Let $(x_{\alpha })$ be a net in a Banach lattice $E$. If $%
(x_{\alpha })$ is almost order bounded or relatively weakly compact, then $%
x_{\alpha }\overset{uaw}{\longrightarrow }0\Leftrightarrow \left\vert
x_{\alpha }\right\vert \overset{w}{\longrightarrow }0$.
\end{proposition}

\begin{proof}
Clearly $\left\vert x_{\alpha }\right\vert \overset{w}{\longrightarrow }0$
implies $x_{\alpha }\overset{uaw}{\longrightarrow }0$. Now assumes that $%
x_{\alpha }\overset{uaw}{\longrightarrow }0$ and pick any $x^{\prime }\in
\left( E^{\prime }\right) ^{+}$ and $\varepsilon >0$.

\begin{itemize}
\item If $(x_{\alpha })$ is almost order bounded then there exists some $%
u\in E^{+}$ such that, for all $\alpha $,%
\begin{equation*}
\left\Vert \left( \left\vert x_{\alpha }\right\vert -u\right)
^{+}\right\Vert <\varepsilon .
\end{equation*}
\end{itemize}

So,%
\begin{eqnarray*}
x^{\prime }\left( \left\vert x_{\alpha }\right\vert \right) &=&x^{\prime
}\left( \left\vert x_{\alpha }\right\vert -u\right) ^{+}+x^{\prime }\left(
\left\vert x_{\alpha }\right\vert \wedge u\right) \\
&\leq &\varepsilon \left\Vert x^{\prime }\right\Vert +x^{\prime }\left(
\left\vert x_{\alpha }\right\vert \wedge u\right) .
\end{eqnarray*}%
Since $x^{\prime }\left( \left\vert x_{\alpha }\right\vert \wedge u\right)
\longrightarrow 0$, we conclude that $x^{\prime }\left( \left\vert x_{\alpha
}\right\vert \right) \longrightarrow 0$.

\begin{itemize}
\item If $(x_{\alpha })$ is relatively weakly compact then, by Theorem 4.37
of \cite{Alip}, there exists $u\in E^{+}$ such that, for all $\alpha $,%
\begin{equation*}
x^{\prime }\left( \left\vert x_{\alpha }\right\vert -\left\vert x_{\alpha
}\right\vert \wedge u\right) =x^{\prime }\left( \left\vert x_{\alpha
}\right\vert -u\right) ^{+}<\varepsilon .
\end{equation*}
\end{itemize}

Since $x^{\prime }\left( \left\vert x_{\alpha }\right\vert \wedge u\right)
\longrightarrow 0$, we conclude that $x^{\prime }\left( \left\vert x_{\alpha
}\right\vert \right) \longrightarrow 0$.

Therefore, $\left\vert x_{\alpha }\right\vert \overset{w}{\longrightarrow }0$%
, as desired.
\end{proof}

\begin{cor}
Let $E$ be a Banach lattice. If $(x_{\alpha })\subset E$ is a disjoint net
and almost order bounded in $E^{\prime \prime }$, then $\left\vert x_{\alpha
}\right\vert \overset{w}{\longrightarrow }0$.
\end{cor}

\begin{proof}
Since $\left( x_{\alpha }\right) $ is a disjoint net in $E^{\prime \prime }$%
, it follows from Lemma 2 of \cite{Zabeti18} that $x_{\alpha }\overset{uaw}{%
\longrightarrow }0$ (as a net in $E^{\prime \prime }$). So, by proposition %
\ref{almostorder}, $\left\vert x_{\alpha }\right\vert \longrightarrow 0$ for
$\sigma \left( E^{\prime \prime },E^{\prime \prime \prime }\right) $, and
hence $\left\vert x_{\alpha }\right\vert \longrightarrow 0$ for $\sigma
\left( E,E^{\prime }\right) $.
\end{proof}

The next result is a $uaw$ variant of \cite[Lemma 9.10]{Kan17}; its proof
follows from Proposition \ref{almostorder}.

\begin{lem}
\label{Lemma-uaw}Let $(x_{n})$ be a sequence in a Banach lattice $E$. If $%
x_{n}\overset{w}{\longrightarrow }x$ and $x_{n}\overset{uaw}{\longrightarrow
}y$ then $x=y$.
\end{lem}

\begin{proof}
Put $z_{n}=x_{n}-y$ and $z=x-y$. Then $z_{n}\overset{w}{\longrightarrow }z$
and $z_{n}\overset{uaw}{\longrightarrow }0$. Since $\left( z_{n}\right) $ is
relatively weakly compact, then by Proposition \ref{almostorder}, we have $%
\left\vert z_{n}\right\vert \overset{w}{\longrightarrow }0$ and hence $z_{n}%
\overset{w}{\longrightarrow }0$. So $z=0$, a desired.
\end{proof}

In \cite{Kan17} (resp. \cite{Zabeti19}), the concept of (sequentially) $un$%
-compact (resp. $uaw$-compact) operator is defined and studied. An operator $%
T\colon X\rightarrow E$, where $X$ is a Banach space and $E$ is a Banach
lattice, is called $un$\emph{-compact} (resp. $uaw$\emph{-compact}) if $%
T(B_{X})$ is relatively $un$-compact (resp. $uaw$-compact), where $B_{X}$
denotes the closed unit ball of $X$. Equivalently, for every norm bounded
net $(x_{\alpha })$ its image has a subnet, which is $un$-convergent (resp. $%
uaw$-convergent).

And $T$ is called \emph{sequentially }$un$\emph{-compact} (resp. \emph{%
sequentially }$uaw$\emph{-compact}) if $T(B_{X})$ is relatively sequentially
$un$-compact (resp. sequentially $uaw$-compact). Equivalently, for every
norm bounded sequence $(x_{n})$ its image has a subsequence, which is $un$%
-convergent (resp. $uaw$-convergent).

In \cite{Kan17}, the authors discussed when weakly compact operators are
sequentially $un$-compact. In the next, we discuss when weakly compact
operators are sequentially $uaw$-compact.

Theorem 7 of \cite{Zabeti18}, Theorem \ref{sequence} and Corollary \ref{cor3}
yield the following.

\begin{proposition}
\label{weaklycompact}Let $T\colon X\rightarrow E$ be an operator from a
Banach space $X$ into a Banach lattice $E$.

\begin{enumerate}
\item If $E^{\prime }$ is order continuous and $T$ is (sequentially) $uaw$%
-compact, then $T$ is weakly compact;

\item If the lattice operations of $E$ are sequentially weakly continuous,
and $T$ is weakly compact, then $T$ is sequentially $uaw$-compact;

\item If the lattice operations of $E$ are sequentially weakly continuous
and $E^{\prime }$ is order continuous, then $T$ is weakly compact iff it is
sequentially $uaw$-compact;

\item If $E$ is discrete and both $E$ and $E^{\prime }$ are order
continuous, then $T\colon X\rightarrow E$ is weakly compact iff it is
(sequentially) $uaw$-compact iff it (sequentially) $un$-compact.
\end{enumerate}
\end{proposition}

\begin{rem}
Note that a weakly compact operator is not necessary sequentially $uaw$%
-compact. In fact, the identity operator $I:L^{2}\left[ 0,1\right]
\rightarrow L^{2}\left[ 0,1\right] $ is weakly compact but it is not
sequentially $uaw$-compact. To see this, let $\left( r_{n}\right) $ be the
sequence of Rademacher function in $L^{2}\left[ 0,1\right] $. This sequence
is order bounded satisfying $r_{n}\overset{w}{\longrightarrow }0$, $%
\left\vert r_{n}\right\vert =\mathbf{1,}$ and $\left\Vert I\left(
r_{n}\right) \right\Vert =\left\Vert r_{n}\right\Vert =1$ (see \cite[p. 196]%
{MN}). So $\left( r_{n}\right) $ has no subsequence $uaw$-convergent.

Also, a (sequentially) $uaw$-compact operator is not necessary weakly
compact. In fact, by Proposition 9.1 of \cite{Kan17}, the identity operator $%
I:\ell ^{1}\rightarrow \ell ^{1}$ is (sequentially) $un$-compact (and hence,
it is (sequentially) $uaw$-compact) but is not weakly compact.
\end{rem}

\begin{cor}
An operator $T\colon X\rightarrow E$ from a Banach space $X$ into an
AM-space $E$ is weakly compact iff it is sequentially $uaw$-compact.
\end{cor}

\begin{theorem}
\label{sequentially}For a Banach lattice $E$, the following assertions are
equivalent:

\begin{enumerate}
\item The lattice operations of $E$ are sequentially weakly continuous;

\item For every Banach space $X$, every weakly compact operator $%
T:X\rightarrow E$ is sequentially $uaw$-compact;

\item Every weakly compact operator $T:\ell ^{1}\rightarrow E$ is
sequentially $uaw$-compact
\end{enumerate}
\end{theorem}

\begin{proof}
$\left( 1\right) \Rightarrow \left( 2\right) $ Follows from Proposition \ref%
{weaklycompact} (2).

$\left( 2\right) \Rightarrow \left( 3\right) $ Obvious.

$\left( 3\right) \Rightarrow \left( 1\right) $ The proof is very similar to
that \cite[Theorem 9.11]{Kan17} (by using Theorem \ref{sequence} and Lemma %
\ref{Lemma-uaw}).
\end{proof}

Recall that an operator $T:E\rightarrow X$, from a Banach lattice into a
Banach space, is called \emph{AM-compact} (resp. \emph{order weakly compact}%
) if it maps order intervals to relatively compact (resp. weakly compact)
sets. It is proved in \cite[Proposition 9.9]{Kan17} that every order bounded
$un$-compact operator is AM-compact. In the next, we prove that this result
is also true for order bounded sequentially compact operators. Also, we
prove that order bounded (sequentially) $uaw$-compact operators are order
weakly compacts.

\begin{proposition}
\label{orderweakly}Let $T:E\rightarrow F$ be an order bounded operator
betwen Banach lattices.

\begin{enumerate}
\item If $T$ is sequentially $un$-compact operator then $T$ is AM-compact.

\item If $T$ is (sequentially) $uaw$-compact operator then $T$ is order
weakly compact.
\end{enumerate}
\end{proposition}

\begin{proof}
1) Assume that $T$ is sequentially $un$-compact. It suffices to show that
every order bounded sequence $(x_{n})$ in $E$ its image has a norm
convergent subsequence. Since $T$ is sequentially $un$-compact, there is a
subsequence $\left( x_{n_{k}}\right) $ such that $T\left( x_{n_{k}}\right)
\overset{un}{\longrightarrow }x$ for some $x\in F$. As $T$ is order bounded,
$\left( T\left( x_{n_{k}}\right) \right) $ is order bounded, and, therefore,
$T\left( x_{n_{k}}\right) \overset{\left\Vert \cdot \right\Vert }{%
\longrightarrow }x$ as desired.

2) We establish the result when $T$ is $uaw$-compact; the other case is
similar. It suffices to show that every order bounded net $\left( x_{\alpha
}\right) $ in $E$ its image has a weak convergent subnet. Since $T$ is $uaw$%
-compact, there is a subnet $\left( y_{\beta }\right) $ of $(x_{\alpha })$
such that $T\left( y_{\beta }\right) \overset{uaw}{\longrightarrow }x$ for
some $x\in F$. As $T$ is order bounded, $\left( T\left( y_{\beta }\right)
\right) $ is order bounded, and, therefore, $T\left( y_{\beta }\right)
\longrightarrow x$ for the absolute weak topology $\left\vert \sigma
\right\vert \left( E,E^{\prime }\right) $. Thus, $T\left( y_{\beta }\right)
\overset{w}{\longrightarrow }x$, as desired.
\end{proof}

Note that the converse of (2) in Proposition \ref{orderweakly} is false: the
identity operator on $c_{0}$ is order weakly compact but is neither $uaw$%
-compact nor sequentially $uaw$-compact (by Proposition 13 of \cite{Zabeti18}%
).

In \cite[Example 9.7]{Kan17}, the authors presented an example to show that,
in general, $un$-compactness is not inherited under domination. In fact,
there exist two operators $S,T:\ell ^{1}\rightarrow L^{1}\left[ 0,1\right] $
satisfying $0\leq T\leq S$ with $S$ is (compact) $un$-compact and $T$ is
neither $un$-compact nor sequentially $un$-compact. Then $S$ is $uaw$%
-compact, and since $L^{1}\left[ 0,1\right] $ is order continuous, it
follows from \cite[Theorem 4]{Zabeti18} that $T$ is neither $uaw$-compact
nor sequentially $uaw$-compact.

\begin{proposition}
Let $S,T:E\rightarrow F$ be two positive operators betwen Banach lattices
satisfying $0\leq S\leq T$.

\begin{enumerate}
\item Suppose that the lattice operations of $F$ are sequentially weakly
continuous and and both $E^{\prime }$ and $F^{\prime }$ are order
continuous. If $T$ is sequentially $uaw$-compact then so is $S$.

\item Suppose that $F$ is discrete and both $F$ and $F^{\prime }$ are order
continuous. If $T$ is (sequentially) $uaw$-compact (resp. $un$-compact) then
so is $S$.
\end{enumerate}
\end{proposition}

\begin{proof}
Follows immediately from Proposition \ref{weaklycompact} and \cite[Theorem
5.31]{Alip}.
\end{proof}

\begin{rem}
The assumption \textquotedblleft $E^{\prime }$ is order
continuous\textquotedblright\ of (1) (resp. \textquotedblleft $F$ is order
continuous\textquotedblright\ of (2)) in the previous result cannot be
removed. Indeed, it follows from \cite[Example 5.30]{Alip} that there exist
two operators $S,T:L^{1}\left[ 0,1\right] \rightarrow \ell ^{\infty }$
satisfying $0\leq S\leq T$ with $T$ is (compact) weakly compact and $S$ is
not weakly compact. Then $T$ is (sequentially) $uaw$-compact, and since $%
\left( \ell ^{\infty }\right) ^{\prime }$ is order continuous, it follows
from Proposition \ref{weaklycompact} (1) that $S$ is neither $uaw$-compact
nor sequentially $uaw$-compact.
\end{rem}

It is proved in \cite[p. 278]{Kan17} that if an operator $T:X\rightarrow E$
is sequentially $un$-compact and semi-compact then $T$ is compact.

\begin{proposition}
If an operator $T:X\rightarrow E$, from a Banach space $X$ into a Banach
lattice $E$, is (sequentially) $uaw$-compact and semi-compact then $T$ is
weakly compact.
\end{proposition}

\begin{proof}
We will prove the statement for the sequential case; the other case is
analogous. Let $(x_{n})$ be a norm bounded sequence in $X$. There is a
subsequence $\left( x_{n_{k}}\right) $ such that $T\left( x_{n_{k}}\right)
\overset{uaw}{\longrightarrow }x$ for some $x\in E$. Since $T$ is
semi-compact, the sequence $\left( T\left( x_{n_{k}}\right) \right) $ is
almost order bounded (and so is $\left( T\left( x_{n_{k}}\right) -x\right) $%
). Therefore, $\left\vert T\left( x_{n_{k}}\right) -x\right\vert \overset{w}{%
\longrightarrow 0}$ (by Proposition \ref{almostorder}). Thus, $T\left(
x_{n_{k}}\right) \overset{w}{\longrightarrow }x$, as desired.
\end{proof}

Following O. Zabeti \cite{Zabeti20}, we shall say that a continuous operator
$T:E\rightarrow X$, where $X$ is a Banach space, is called \emph{unbounded
continuous }if for each bounded net $(x_{\alpha })\subset E$, $x_{\alpha }%
\overset{uaw}{\longrightarrow }0$ implies $T\left( x_{\alpha }\right)
\overset{w}{\longrightarrow }0$. It is \emph{sequentially unbounded
continuous} provided that the property happens for sequences. Moreover, a
continuous operator $T:E\rightarrow F$ between Banach lattices, is said to
be $uaw$\emph{-continuous }(\emph{sequentially }$uaw$-continuous\emph{)} if $%
T$ maps every norm bounded $uaw$-null net (sequence) into a $uaw$-null net
(sequence).

The following result tells us when every continuous operator is unbounded
continuous.

\begin{theorem}
\label{unbounded-cont}Let $E$ be a Banach lattice and $X$ be a non-zero
Banach space. The following are equivalent:

\begin{enumerate}
\item $E^{\prime }$ is order continuous;

\item Every continuous operator $T:E\rightarrow X$ is unbounded continuous;

\item Every continuous operator $T:E\rightarrow X$ is sequentially unbounded
continuous.
\end{enumerate}
\end{theorem}

\begin{proof}
$\left( 1\right) \Rightarrow \left( 2\right) $ Follows immediately from
Theorem 7 of \cite{Zabeti18}.

$\left( 2\right) \Rightarrow \left( 3\right) $ Obvious.

$\left( 3\right) \Rightarrow \left( 1\right) $ Assume that $E^{\prime }$ is
not order continuous. To finish the proof, we have to construct a continuous
operator $T:E\longrightarrow X$ which is not sequentially unbounded
continuous.

Since the norm of $E^{\prime }$\ is not order continuous, it follows from
Theorem 116.1 of \cite{Z2} that there is a norm bounded disjoint sequence $%
\left( u_{n}\right) $ of positive elements in $E$ which does not converge
weakly to zero. Hence, we may assume that $\left\Vert u_{n}\right\Vert \leq
1 $ for all $n$ and also that for some $0\leq \varphi \in E^{\prime }$ and
some $\varepsilon >0$ we have $\varphi \left( u_{n}\right) >\varepsilon $
for all $n$. Moreover, by Theorem 116.3 (i) of \cite{Z2}, the components $%
\varphi _{n}$ of $\varphi $ in the carrier $C_{u_{n}}$ form an order bounded
disjoint sequence in $\left( E^{\prime }\right) ^{+}$ such that
\begin{equation*}
\varphi _{n}\left( u_{n}\right) =\varphi \left( u_{n}\right) \text{ for all }%
n\text{ and }\varphi _{n}\left( u_{m}\right) =0\text{ if }n\neq m.\eqno{(*)}
\end{equation*}%
Note that $0\leq \varphi _{n}\leq \varphi $ holds for all $n$.

On the other hand, let $y$ be a non-zero vector in $X$ \ and consider the
operator $T:E\rightarrow X$ defined by defined by%
\begin{equation*}
T\left( x\right) =\left( \sum_{n=1}^{\infty }\frac{\varphi _{n}(x)}{\varphi
\left( u_{n}\right) }\right) y
\end{equation*}%
for all $x\in E$. Since
\begin{equation*}
\sum_{n=1}^{\infty }\left\vert \frac{\varphi _{n}(x)}{\varphi \left(
u_{n}\right) }\right\vert \leq \frac{1}{\varepsilon }\sum_{n=1}^{\infty
}\varphi _{n}(\left\vert x\right\vert )\leq \frac{1}{\varepsilon }\varphi
(\left\vert x\right\vert )
\end{equation*}%
holds for each $x\in E$, the operator $T$ is well defined.

We claim that $T$ is not sequentially unbounded continuous. To see this,
note that $\left( u_{n}\right) $ is a norm bounded disjoint sequence and $%
u_{n}\overset{uaw}{\longrightarrow }0$ (by Lemma 2 of \cite{Zabeti18}). But,
from $(\ast )$ we have $T(u_{n})=y\overset{w}{\nrightarrow }0$. Thus $T$ is
not sequentially unbounded continuous, and this ends the proof of the
Theorem.
\end{proof}

\begin{rem}
The previous result improve Theorem 1 of \cite{Zabeti20}. In fact, if every
continuous operator $T:\ell ^{1}\rightarrow X$ is (sequentially) unbounded
continuous then $X=\left\{ 0\right\} $.
\end{rem}

\begin{cor}
For a Banach lattice $E$, the following are equivalent:

\begin{enumerate}
\item Every continuous operator $T:E\rightarrow E$ is unbounded continuous;

\item $E^{\prime }$ is order continuous.
\end{enumerate}
\end{cor}

The following result follows immediately from Theorems \ref{sequence} and %
\ref{unbounded-cont}.

\begin{cor}
Let $E$ and $F$ be two Banach lattices. If $E^{\prime }$ is order continuous
and the lattice operations of $F$ are weakly sequentially continuous, then
every continuous operator $T:E\rightarrow F$ is\ sequentially $uaw$%
-continuous.
\end{cor}

The following result follows immediately from Corollary \ref{cor1} and it
generalize Corollary 4 of \cite{Zabeti20}.

\begin{proposition}
Let $E$ and $F$ be two Banach lattices. If $F^{\prime }$ is order continuous
and the lattice operations of $F$ are weakly sequentially continuous, then a
continuous operator $T:E\rightarrow F$ is\ sequentially unbounded continuous
if and only if it is sequentially $uaw$-continuous.
\end{proposition}

\begin{theorem}
Let $E$ be a Banach lattice and $F$ be a non-zero Banach lattice. If every
continuous operator $T:E\rightarrow F$ is (sequentially) $uaw$-continuous
then $E^{\prime }$ is order continuous.
\end{theorem}

\begin{proof}
The proof is similar to that $\left( 3\right) \Rightarrow \left( 1\right) $
of Theorem \ref{unbounded-cont}.
\end{proof}

\begin{rem}
The previous result improve Theorem 7 of \cite{Zabeti20}. In fact, if every
continuous operator $T:L_{1}\left[ 0,1\right] \rightarrow F$ is
(sequentially) $uaw$-continuous then $F=\left\{ 0\right\} $.
\end{rem}

An operator $T$ from a Banach lattice $E$ into a Banach space $X$ is said to
be $b$-weakly compact, if it maps each subset of $E$ which is $b$-order
bounded (i.e. order bounded in the topological bidual $E^{\prime \prime }$)
into a relatively weakly compact subset in $X$ \cite{Alpay1}. Several
interesting characterizations of this class of operators are given in \cite%
{Alpay1, Alpay3, Alpay4, Aqz1}.

To prove next result, we need the following lemma.

\begin{lem}
\label{Aqz1}Let $T$ be a $b$-weakly compact operator from a Banach lattice $%
E $ into a Banach space $X$. If $(x_{n})$ is a $b$-order bounded sequence of
$E $, then for each $\varepsilon >0$ there exists some $u\in E^{+}$ such that%
\begin{equation*}
q_{T}((|x_{n}|-u)^{+})\leq \varepsilon
\end{equation*}%
holds for all $n$, where $q_{T}$ is the lattice seminorm on $E$ defined by%
\begin{equation*}
q_{T}\left( x\right) =\sup \left\{ \left\Vert T\left( y\right) \right\Vert
:\left\vert y\right\vert \leq \left\vert x\right\vert \right\}
\end{equation*}%
for each $x\in E.$
\end{lem}

In the following, we give our characterizations of $b$-weakly compact
operators using the $uaw$-convergence sequences.

\begin{theorem}
\label{b-weaklycompact}Let $T$ be an operator from a Banach lattice $E$ into
a Banach space $X$.Then the following assertions are equivalent:

\begin{enumerate}
\item $T$ is $b$-weakly compact;

\item For every $b$-order bounded sequence $(x_{n})\subset E^{+}$, $x_{n}%
\overset{uaw}{\longrightarrow }0$ implies $q_{T}(x_{n})\rightarrow 0$;

\item For every $b$-order bounded sequence $(x_{n})\subset E^{+}$, $x_{n}%
\overset{uaw}{\longrightarrow }0$ implies $\left\Vert T(x_{n})\right\Vert
\rightarrow 0$;

\item For every $b$-order bounded disjoint sequence $(x_{n})\subset E^{+}$, $%
\left\Vert T(x_{n})\right\Vert \rightarrow 0$.
\end{enumerate}
\end{theorem}

\begin{proof}
$\left( 1\right) \Leftrightarrow \left( 4\right) $ Follows from Proposition
2.8 of \cite{Alpay1}.

$\left( 1\right) \Rightarrow \left( 2\right) $ Assume that $T$ is $b$-weakly
compact and let $(x_{n})\subset E^{+}$ be a $b$-order bounded sequence such
that $x_{n}\overset{uaw}{\longrightarrow }0$.\ Pick $\varepsilon >0$. So, by
Lemma \ref{Aqz1}, there exists some $u\in E^{+}$ such that
\begin{equation*}
q_{T}\left( x_{n}-x_{n}\wedge u\right) =q_{T}\left( \left( x_{n}-u\right)
^{+}\right) \leq \varepsilon
\end{equation*}%
holds for all $n$. On the other hand, since $\left( x_{n}\wedge u\right)
\subset E^{+}$ is order bounded and weakly null, it follows from\ Theorem
2.2 of \cite{Aqz1} that $q_{T}\left( x_{n}\wedge u\right) \rightarrow 0$.
Therefore,
\begin{equation*}
q_{T}\left( x_{n}\right) \leq \varepsilon +q_{T}\left( x_{n}\wedge u\right)
\leq 2\varepsilon
\end{equation*}%
for all sufficiently large $n$. So $q_{T}\left( x_{n}\right) \rightarrow 0$.

$\left( 2\right) \Rightarrow \left( 3\right) $ Follows immediately from the
inequalities $\left\Vert T\left( x\right) \right\Vert \leq q_{T}(x)$ holds
for all $x\in E$.

$\left( 3\right) \Rightarrow \left( 4\right) $ Follows immediately from the
fact that for every ($b$-order bounded) disjoint sequence $(x_{n})\subset
E^{+}$, we have $x_{n}\overset{uaw}{\longrightarrow }0$ (Lemma 2 of \cite%
{Zabeti18}).
\end{proof}

It is proved in \cite[Proposition 2.10]{Alpay1} that $E$ is a KB-space if
and only if the identity operator $\mathrm{Id}_{E}:E\rightarrow E$ is a $b$%
-weakly compact operator.

\begin{cor}
For a Banach lattice $E$, the following assertions are equivalent:

\begin{enumerate}
\item $E$ is a KB-space;

\item The identity operator $\mathrm{Id}_{E}:E\rightarrow E$ is $b$-weakly
compact;

\item For every $b$-order bounded sequence $(x_{n})\subset E^{+}$, $x_{n}%
\overset{uaw}{\longrightarrow }0$ implies $\left\Vert x_{n}\right\Vert
\rightarrow 0$.
\end{enumerate}
\end{cor}

\end{document}